\documentclass[reqno,12pt,letterpaper]{amsart}
\usepackage[proof]{sdmacros}
\title[Control of eigenfunctions on hyperbolic surfaces]%
{Control of eigenfunctions on hyperbolic surfaces:\\
an application of fractal uncertainty principle}
\author{Semyon Dyatlov}
\email{dyatlov@math.mit.edu}
\address{Department of Mathematics, Massachusetts Institute of Technology,
77 Massachusetts Ave, Cambridge, MA 02139}
\begin{document}

\begin{abstract}
This expository article,
written for the proceedings of the Journ\'ees EDP (Roscoff, June 2017),
presents recent work joint with Jean Bourgain~\cite{fullgap}
and Long Jin~\cite{meassupp}. We in particular show that eigenfunctions of the Laplacian
on hyperbolic surfaces are bounded from below in $L^2$ norm on each nonempty
open set, by a constant depending on the set but not on the eigenvalue.
\end{abstract}

\maketitle

%%%%%%%%%%%%%%%%%%%%%%%%%%%%%%%%%%%%%%%%%%%%%%%%%%%%%%%%%%%%%%%%%%%%%%%%%%%%%%%%
%                                 INTRODUCTION                                 %
%%%%%%%%%%%%%%%%%%%%%%%%%%%%%%%%%%%%%%%%%%%%%%%%%%%%%%%%%%%%%%%%%%%%%%%%%%%%%%%%
\section{Introduction}

The purpose of this expository article is to present a recent result of Dyatlov--Jin~\cite{meassupp}
on \emph{control of eigenfunctions on hyperbolic surfaces}. The main tool is
the \emph{fractal uncertainty principle} proved by Bourgain--Dyatlov~\cite{fullgap}.
To keep the article short we will omit many technical details of the proofs,
referring the reader to the original papers~\cite{fullgap,meassupp} and to the lecture notes~\cite{fupnotes}.
The results discussed here belong to \emph{quantum chaos};
see~\cite{MarklofReview,ZelditchReview,SarnakQUE} for introductions to the subject.

The setting considered is a \emph{hyperbolic surface},
that is a compact Riemannian surface $(M,g)$ of Gauss curvature $-1$,
which we fix throughout the paper. For convenience we assume that $M$ is connected and
oriented. The homogeneous \emph{geodesic flow}
\begin{equation}
  \label{e:geodesic-flow}
\varphi_t:T^*M\setminus 0\to T^*M\setminus 0,\quad
T^*M\setminus 0:=\{(x,\xi)\colon x\in M,\ \xi\in T_x^*M,\ \xi\neq 0\}  
\end{equation}
is a classical example of a strongly chaotic, or more precisely, \emph{hyperbolic},%
\footnote{The word `hyperbolic' is used in two different meanings here: a surface is hyperbolic
when it has constant negative curvature and a flow is hyperbolic when it admits a stable/unstable
decomposition. Luckily for us, hyperbolic surfaces do give rise to hyperbolic geodesic flows.}
dynamical system, see~\S\ref{s:long-time} below. We can view $\varphi_t$ as a Hamiltonian flow:
\begin{equation}
  \label{e:p-def}
\varphi_t=\exp(tH_p),\quad
p(x,\xi)=|\xi|_g.
\end{equation}
Let $-\Delta_g\geq 0$ be the Laplace--Beltrami operator on $M$.
Our first result is
%%%%%%%%%%%%%%%%%%%%%%%%%%%%%%%%%%%%%%%%%%%%%%%%%%%%%%%%%%%%%%%%%%%%%%%%%%%%%%%%
\begin{theo}
  \label{t:basic-control}
Fix a nonempty open set $\Omega\subset M$. Then for each eigenfunction of $\Delta_g$
\begin{equation}
  \label{e:eigenfunction}
u\in C^\infty(M),\quad
(-\Delta_g-\lambda^2)u=0,\quad
\|u\|_{L^2(M)}=1
\end{equation}
we have the following lower bound
where the constant $c_\Omega$ depends on $\Omega$ but not on $\lambda$:
\begin{equation}
  \label{e:basic-control}
\|u\|_{L^2(\Omega)}\geq c_\Omega>0.
\end{equation}
\end{theo}
%%%%%%%%%%%%%%%%%%%%%%%%%%%%%%%%%%%%%%%%%%%%%%%%%%%%%%%%%%%%%%%%%%%%%%%%%%%%%%%%
\Remarks 1. For bounded $\lambda$, \eqref{e:basic-control}
follows from unique continuation principle. Thus Theorem~\ref{t:basic-control}
is really a result about the \emph{high frequency limit} $\lambda\to \infty$.

\noindent 2. Theorem~\ref{t:basic-control} uses global information about the structure
of the manifold $M$; it does not hold for every Riemannian manifold. For instance,
if $M$ is the round 2-sphere and $\Omega$ lies in just one hemisphere, then
one can construct a sequence of eigenfunctions $u_n$,
with $\|u_n\|_{L^2(\Omega)}$ exponentially small as $\lambda_n\to\infty$.

An application of Theorem~\ref{t:basic-control} (strictly speaking, of
Theorem~\ref{t:advanced-control} below) is the following observability result
for the Schr\"odinger equation:
%%%%%%%%%%%%%%%%%%%%%%%%%%%%%%%%%%%%%%%%%%%%%%%%%%%%%%%%%%%%%%%%%%%%%%%%%%%%%%%%
\begin{theo}\cite{JinControl}
Let $M,\Omega$ be as in Theorem~\ref{t:basic-control} and fix $T>0$. Then there exists
$C=C(\Omega,T)>0$ such that for all $f\in L^2(M)$ we have
$$
\|f\|_{L^2(M)}^2\leq C\int_0^T \|e^{it\Delta}f\|_{L^2(\Omega)}^2\,dt.
$$
\end{theo}
%%%%%%%%%%%%%%%%%%%%%%%%%%%%%%%%%%%%%%%%%%%%%%%%%%%%%%%%%%%%%%%%%%%%%%%%%%%%%%%%

%%%%%%%%%%%%%%%%%%%%%%%%%%%%%%%%%%%%%%%%%%%%%%%%%%%%%%%%%%%%%%%%%%%%%%%%%%%%%%%%
\subsection{A semiclassical result}

We now give a stronger version of Theorem~\ref{t:basic-control}
(see Theorem~\ref{t:advanced-control} below) which
allows to localize $u$ in frequency in addition to localization in position.
To state it we introduce a \emph{semiclassical quantization}
$$
\Op_h:a\in C_0^\infty(T^*M)\ \mapsto\ \Op_h(a):L^2(M)\to L^2(M).
$$
We refer the reader to~\cite[\S14.2.2]{e-z} for the definition and properties of $\Op_h$.
The operator $\Op_h(a)$, acting on appropriately chosen semiclassical Sobolev
spaces, can also be defined when $a$ is not compactly supported but instead
satisfies a polynomial growth bound as $\xi\to\infty$. Functions satisfying
such bounds are called \emph{symbols}.

We henceforth put $h:=\lambda^{-1}>0$, so that the equation~\eqref{e:eigenfunction}
becomes
\begin{equation}
  \label{e:beigen}
(-h^2\Delta_g-1)u=0,\quad
\|u\|_{L^2(M)}=1.
\end{equation}
The elliptic estimate implies (see~\eqref{e:elliptor} below)
that solutions to~\eqref{e:beigen} are microlocalized on the cosphere bundle
$S^*M=\{(x,\xi)\colon p(x,\xi)=1\}$:
\begin{equation}
  \label{e:basic-ellipticity}
a|_{S^*M}\equiv 0\ \Longrightarrow\ \|\Op_h(a)u\|_{L^2}=\mathcal O(h).
\end{equation}
In particular, $\|\Op_h(a)u\|$ cannot be bounded away from 0 as $h\to 0$
when $a|_{S^*M}\equiv 0$. Our next result in particular implies that such
a lower bound holds for all other $a$:
%%%%%%%%%%%%%%%%%%%%%%%%%%%%%%%%%%%%%%%%%%%%%%%%%%%%%%%%%%%%%%%%%%%%%%%%%%%%%%%%
\begin{theo}\cite[Theorem~2]{meassupp}
  \label{t:advanced-control}
Assume that $a\in C_0^\infty(T^*M)$ and $a|_{S^*M}\not\equiv 0$. Then there exist
constants $C,h_0>0$ depending on $a$ but not on $h$, such that
for all $u\in H^2(M)$ and $h<h_0$
\begin{equation}
  \label{e:advanced-control}
\|u\|_{L^2(M)}\leq C\|\Op_h(a)u\|_{L^2(M)}+{C\log(1/h)\over h}\|(-h^2\Delta_g-1)u\|_{L^2(M)}.
\end{equation}
\end{theo}
%%%%%%%%%%%%%%%%%%%%%%%%%%%%%%%%%%%%%%%%%%%%%%%%%%%%%%%%%%%%%%%%%%%%%%%%%%%%%%%%
\Remarks 1. Theorem~\ref{t:advanced-control} applies to $u$ which are not exact eigenfunctions
of $\Delta_g$. It gives a nontrivial statement when $u$ is an $o(h/\log(1/h))$-quasimode
for the Laplacian.

\noindent 2. Theorem~\ref{t:basic-control} follows immediately from Theorem~\ref{t:advanced-control}.
Indeed, take $a=a(x)\in C_0^\infty(M)$ such that $a\not\equiv 0$ and $\supp a\subset \Omega$.
Then \eqref{e:beigen}, \eqref{e:basic-ellipticity}, and~\eqref{e:advanced-control} together imply that
$1=\|u\|_{L^2(M)}\leq C\|\Op_h(a)u\|_{L^2(M)}$. It remains to use that $\Op_h(a)$ is a multiplication
operator, namely
$\Op_h(a)u=au$.

%%%%%%%%%%%%%%%%%%%%%%%%%%%%%%%%%%%%%%%%%%%%%%%%%%%%%%%%%%%%%%%%%%%%%%%%%%%%%%%%
\subsection{Application to semiclassical measures}

We next present an application of Theorem~\ref{t:advanced-control} to semiclassical defect measures,
which are weak limits of high frequency sequences of eigenfunctions of $\Delta_g$
defined as follows:
%%%%%%%%%%%%%%%%%%%%%%%%%%%%%%%%%%%%%%%%%%%%%%%%%%%%%%%%%%%%%%%%%%%%%%%%%%%%%%%%
\begin{defi}
  \label{d:measure}
Consider a sequence of eigenfunctions
\begin{equation}
  \label{e:measure-sequence}
u_j\in C^\infty(M),\quad
(-h_j^2\Delta_g-1)u_j=0,\quad
\|u_j\|_{L^2(M)}=1,\quad
h_j\to 0.
\end{equation}
Let $\mu$ be a Borel measure on $T^*M$. We say that $u_j$ converges to $\mu$ weakly
if
\begin{equation}
  \label{e:measure-def}
\langle\Op_{h_j}(a)u_j,u_j\rangle_{L^2(M)}\to\int_{T^*M}a\,d\mu\quad\text{as }j\to\infty
\quad\text{for all }a\in C_0^\infty(T^*M).
\end{equation}
We say that $\mu$ is a \textbf{semiclassical measure} on $M$, if $\mu$ is the weak limit
of some sequence of eigenfunctions.
\end{defi}
%%%%%%%%%%%%%%%%%%%%%%%%%%%%%%%%%%%%%%%%%%%%%%%%%%%%%%%%%%%%%%%%%%%%%%%%%%%%%%%%
Semiclassical measures enjoy many nice properties~\cite[Chapter~5]{e-z},
in particular
%%%%%%%%%%%%%%%%%%%%%%%%%%%%%%%%%%%%%%%%%%%%%%%%%%%%%%%%%%%%%%%%%%%%%%%%%%%%%%%%
\begin{itemize}
\item every sequence of eigenfunctions has a subsequence converging weakly to some
measure;
\item every semiclassical measure is a probability measure supported on $S^*M$;
\item every semiclassical measure is invariant under the geodesic flow $\varphi_t$.
\end{itemize}
%%%%%%%%%%%%%%%%%%%%%%%%%%%%%%%%%%%%%%%%%%%%%%%%%%%%%%%%%%%%%%%%%%%%%%%%%%%%%%%%
The question of which $\varphi_t$-invariant probability measures on $S^*M$
can be semiclassical measures has received considerable attention in the quantum chaos
literature. Here we only mention a
few works, referring the reader to the introduction to~\cite{meassupp} for
a more comprehensive overview.

The \emph{Quantum Ergodicity (QE)} theorem of Shnirelman, Zelditch,
and Colin de Verdi\`ere \cite{Shnirelman,ZelditchQE,CdV} implies that a density 1 sequence of eigenfunctions
converges to the \emph{Liouville measure} $\mu_L$, which is the natural volume measure on $S^*M$.
Rundick--Sarnak~\cite{RudnickSarnak} made the \emph{Quantum Unique Ergodicity (QUE)}
conjecture that $\mu_L$ is the only semiclassical measure, i.e. it is the weak limit
of the entire sequence of eigenfunctions.
So far QUE has only been proved for Hecke eigenfunctions on arithmetic surfaces,
by Lindenstrauss~\cite{Lindenstrauss}.

While QUE for general hyperbolic surfaces is still an open problem,
one can show that semiclassical measures satisfy lower bounds
on entropy. In particular, Anantharaman--Nonnenmacher~\cite{AnantharamanNonnenmacher}
proved that each semiclassical measure $\mu$ has Kolmogorov--Sinai entropy
$H_{\mathrm{KS}}(\mu)\geq {1\over 2}$.
This rules out many possible weak limits,
in particular the delta measure on a closed geodesic (which has entropy 0).
We remark that $H_{\mathrm{KS}}(\mu_L)=1$, so in some sense~\cite{AnantharamanNonnenmacher}
rules out half of the invariant measures as weak limits.

For $a\in C_0^\infty(T^*M)$, we have $\|\Op_h(a)u\|_{L^2}^2=\langle \Op_h(|a|^2)u,u\rangle_{L^2}
+\mathcal O(h)\|u\|_{L^2}^2$. Thus Theorem~\ref{t:advanced-control} implies
%%%%%%%%%%%%%%%%%%%%%%%%%%%%%%%%%%%%%%%%%%%%%%%%%%%%%%%%%%%%%%%%%%%%%%%%%%%%%%%%
\begin{theo}
  \label{t:measures}
Any semiclassical measure $\mu$ on $M$ has support equal to the entire $S^*M$,
that is $\mu(U)>0$ for any nonempty open $U\subset S^*M$.  
\end{theo}
%%%%%%%%%%%%%%%%%%%%%%%%%%%%%%%%%%%%%%%%%%%%%%%%%%%%%%%%%%%%%%%%%%%%%%%%%%%%%%%%
We remark that Theorem~\ref{t:measures} is in some sense orthogonal to the entropy
bound mentioned above. In particular, there exist $\varphi_t$-invariant probability
measures $\mu$ on~$S^*M$ such that $\supp\mu\neq S^*M$ but $H_{\mathrm{KS}}(\mu)>{1\over 2}$,
in fact the entropy of $\mu$ may be arbitrarily close to~1. The simplest way to construct
these is to choose $M$ which has a short closed geodesic, cut $M$ along this geodesic
and attach two funnels to obtain a convex co-compact surface $\widetilde M$, and define
$\mu$ as the Patterson--Sullivan measure on the set of trapped geodesics on $\widetilde M$
(see for instance~\cite[\S14.2]{BorthwickBook}).

%%%%%%%%%%%%%%%%%%%%%%%%%%%%%%%%%%%%%%%%%%%%%%%%%%%%%%%%%%%%%%%%%%%%%%%%%%%%%%%%
%%%%%%%%%%%%%%%%%%%%%%%%%%%%%%%%%%%%%%%%%%%%%%%%%%%%%%%%%%%%%%%%%%%%%%%%%%%%%%%%
\section{Propagation of control}

In this section we explain how to reduce Theorem~\ref{t:advanced-control} to a
norm bound~\eqref{e:fup-used} which will follow from the fractal uncertainty principle.

In this section $\|\bullet\|$ denotes the $L^2(M)$-norm. We will only show a weaker
version of the estimate~\eqref{e:advanced-control}, in the case of exact eigenfunctions:
\begin{equation}
  \label{e:lazy-control}
(-h^2\Delta_g-1)u=0\quad\Longrightarrow\quad
\|u\|\leq C\log(1/h)\|\Op_h(a)u\|.
\end{equation}
See the end of~\S\ref{s:proof} for a discussion of how to modify the proof
to remove the $\log(1/h)$ prefactor.

We start with a few basic observations and reductions.
We henceforth assume that $(-h^2\Delta_g-1)u=0$,
$\|u\|=1$, and $h$ is small.

By the semiclassical elliptic estimate,
we have for any two symbols $a,b$ (with the order of growth of $b$ in $\xi$ bounded by that of $a$)
\begin{equation}
  \label{e:elliptor}
\supp b\subset \{a\neq 0\}\quad\Longrightarrow\quad
\|\Op_h(b)u\|\leq C\|\Op_h(a)u\|+\mathcal O(h)\|u\|.
\end{equation}
This follows immediately from the algebraic properties of semiclassical
quantization, writing $\Op_h(b)=\Op_h(b/a)\Op_h(a)+\mathcal O(h)$.
This argument also gives the bound~\eqref{e:basic-ellipticity},
using that $-h^2\Delta_g-1=\Op_h(p^2-1)+\mathcal O(h)$ where $p$ is defined in~\eqref{e:p-def}.

Now, fix $a\in C_0^\infty(T^*M)$ such that $a|_{S^*M}\not\equiv 0$. By~\eqref{e:basic-ellipticity}
the estimate~\eqref{e:lazy-control} only depends on the values of $a$ on $S^*M$. By~\eqref{e:elliptor}
it suffices to prove~\eqref{e:lazy-control} with $a$ replaced by some symbol whose support
lies in $\{a\neq 0\}$. Thus we may assume without loss of generality that
$a$ is homogeneous of degree 0 near $S^*M$ and
\begin{equation}
  \label{e:U-def}
a\equiv 1\quad\text{near }\mathcal U\cap S^*M
\end{equation}
where $\mathcal U\subset T^*M\setminus 0$ is some nonempty open conic set.

%%%%%%%%%%%%%%%%%%%%%%%%%%%%%%%%%%%%%%%%%%%%%%%%%%%%%%%%%%%%%%%%%%%%%%%%%%%%%%%%
\subsection{Partitions and control}

Define $a_1,a_2\in C_0^\infty(T^*M\setminus 0;[0,1])$ such that
\begin{equation}
  \label{e:a-j-def}
a_1=a\quad\text{and}\quad a_1+a_2=1\quad\text{on}\quad S^*M;\quad
\supp a_2\cap \mathcal U=\emptyset.
\end{equation}
We may also assume that $a_1+a_2\leq 1$ everywhere.
Define the pseudodifferential operators
$$
A_j:=\Op_h(a_j):L^2(M)\to L^2(M),\quad
j=1,2.
$$
Note that the $L^2\to L^2$ operator norms of $A_1,A_2,A_1+A_2$ are bounded by $1+\mathcal O(h)$,
see~\cite[(2.10)]{meassupp}. We can choose the quantizations $A_1,A_2$ so that $A_1+A_2$ commutes
with $-\Delta_g$, see~\cite[\S3.1]{meassupp}. By~\eqref{e:basic-ellipticity}
we have
\begin{gather}
  \label{e:part-1}
\|u-(A_1+A_2)u\|=\mathcal O(h),\\
  \label{e:part-2}
\|A_1u\|\leq \|\Op_h(a)u\|+\mathcal O(h).
\end{gather}
The bound~\eqref{e:part-2} says that $\|A_1u\|$ is \emph{controlled}
in the sense that it is bounded by an expression of the form
$C\|\Op_h(a)u\|+o(1)_{h\to 0}$. Our goal is to obtain control
on $u$ in a large region of the phase space $T^*M$.

Since $u$ is an eigenfunction of the Laplacian, it is also an eigenfunction
of the half-wave propagator:
$$
U(t)u=e^{-it/h}u,\quad
U(t):=\exp(-it\sqrt{-\Delta_g}):L^2(M)\to L^2(M).
$$
For a bounded operator $A:L^2(M)\to L^2(M)$, denote
\begin{equation}
  \label{e:A-t-def}
A(t):=U(-t)AU(t).
\end{equation}
Then~\eqref{e:part-2} implies that $A_1(t)u$ is controlled for all $t$:
\begin{equation}
  \label{e:time-control}
\|A_1(t)u\|=\|A_1u\|\leq \|\Op_h(a)u\|+\mathcal O(h).
\end{equation}
On the other hand we have \emph{Egorov's Theorem}~\cite[Theorem~11.1]{e-z}: for bounded $t$,
\begin{equation}
  \label{e:basic-egorov}
A_1(t)=\Op_h(a_1\circ\varphi_t)+\mathcal O(h)_{L^2\to L^2}.
\end{equation}
Therefore, $\|\Op_h(a_1\circ\varphi_t)u\|$ is controlled for any bounded $t$.

The above observations imply the required bound~\eqref{e:lazy-control},
even without the $\log(1/h)$ prefactor,
in cases when $\mathcal U$ satisfies the \emph{geometric control condition},
namely there exists $T>0$ such that for each $(x,\xi)\in S^*M$,
the geodesic segment $\{\varphi_t(x,\xi)\mid 0\leq t\leq T\}$
intersects $\mathcal U$. Indeed, in this case there exists a finite
set of times $t_1,\dots,t_N\in [0,T]$ such that
$$
\widetilde a:=\sum_{\ell=1}^N a_1\circ\varphi_{t_\ell}>0\quad\text{on }S^*M.
$$
We control each $\|\Op_h(a_1\circ\varphi_{t_\ell})u\|$ and thus
the sum $\|\Op_h(\widetilde a)u\|$. By~\eqref{e:basic-ellipticity}
and~\eqref{e:elliptor} we have
$1=\|u\|\leq C\|\Op_h(\widetilde a)u\|+\mathcal O(h)$, giving~\eqref{e:lazy-control}.

%%%%%%%%%%%%%%%%%%%%%%%%%%%%%%%%%%%%%%%%%%%%%%%%%%%%%%%%%%%%%%%%%%%%%%%%%%%%%%%%
\subsection{Long time propagation}
  \label{s:long-time}

Many sets $\mathcal U\subset S^*M$ do not satisfy the geometric control condition
(for instance it is enough for $\mathcal U$ to miss just one closed geodesic on~$M$).
To obtain~\eqref{e:lazy-control} we then need to use Egorov's Theorem~\eqref{e:basic-egorov}
(more precisely, its version for $A_2$) for times $t$ which grow as $h\to 0$.

One has to take extra caution:
for $t$ which are too large, the derivatives of the symbol $a_2\circ\varphi_t$
grow too fast with $h$ and the quantization $\Op_h(a_2\circ\varphi_t)$ no longer makes
sense (i.e. no longer has the standard properties). One can also think of that
problem as $a_2\circ\varphi_t$ localizing to a set which is so thin in some directions
that such localization contradicts the uncertainty principle.

It turns out that for hyperbolic surfaces, Egorov's Theorem still holds
when $|t|\leq \rho\log(1/h)$ for any fixed $0\leq\rho<1$. To explain this,
we first study the growth of derivatives of $a_2\circ\varphi_t$ as $|t|\to\infty$.
The geodesic flow $\varphi_t$ is \emph{hyperbolic} on the level sets $\{p=\const\}$,
namely it has a stable/unstable decomposition. More precisely, there exists a smooth frame
$H_p,U_+,U_-,D$ (see~\cite[\S 2.1]{meassupp})
on $T^*M\setminus 0$
such that:
\begin{itemize}
\item $H_p$ is the generator of the flow, that is $\varphi_t=\exp(tH_p)$;
\item $D=\xi\cdot\partial_\xi$ is the generator of dilations in the fibers of $T^*M\setminus 0$,
and it is invariant under $\varphi_t$: $(\varphi_t)_*D=D$;
\item $U_+$ is the \emph{stable horocyclic field}, which decays exponentially along the flow:
$(\varphi_t)_*U_+=e^{-t}U_+$;
\item $U_+$ is the \emph{unstable horocyclic field}, which grows exponentially along the flow:
$(\varphi_t)_*U_-=e^{t}U_-$.
\end{itemize}
We see that $H_p(a_2\circ\varphi_t)$ and $D(a_2\circ\varphi_t)$ are bounded uniformly in $t$;
$U_+(a_2\circ \varphi_t)$ and $U_-(a_2\circ\varphi_{-t})$ are bounded when $t\geq 0$;
and $U_-(a_2\circ\varphi_t)$ and $U_+(a_2\circ\varphi_{-t})$ grow like $e^t$ when $t\geq 0$.
In particular we have (see~\cite[\S 2.3]{meassupp})
$$
\begin{aligned}
a_2\circ\varphi_t\in S^{\comp}_{L_s,\rho}(T^*M\setminus 0),&\quad 0\leq t\leq \rho\log(1/h);\\
a_2\circ\varphi_t\in S^{\comp}_{L_u,\rho}(T^*M\setminus 0),&\quad -\rho\log(1/h)\leq t\leq 0
\end{aligned}
$$
where $L_s,L_u$ are the \emph{weak stable/unstable} Lagrangian foliations with tangent spaces
$$
L_s=H_p\oplus U_+,\quad
L_u=H_p\oplus U_-
$$
and the class $S^{\comp}_{L,\rho}(T^*M\setminus 0)$, $L\in \{L_u,L_s\}$, consists of symbols $b(x,\xi;h)$ such that
\begin{itemize}
\item $b$ is compactly supported in $(x,\xi)$ inside an $h$-independent subset of $T^*M\setminus 0$;
\item $\sup |b|\leq C$, where $C$ is independent of $h$;
\item for each vector fields $Y_1,\dots,Y_m,Z_1,\dots,Z_k$ on $T^*M\setminus 0$
such that $Y_1,\dots,Y_m$ are tangent to $L$, and each $\varepsilon>0$ we have
$$
\sup |Y_1\cdots Y_mZ_1\cdots Z_kb|\leq Ch^{-\rho k-\varepsilon}.
$$
\end{itemize}
In other words, $b$ only grows by an arbitrarily small power of $h$ when differentiated
along $L$, but may grow by a factor up to $h^{-\rho-}$ when differentiated in other directions.

Each symbol $b\in S^{\comp}_{L,\rho}(T^*M\setminus 0)$ can be quantized
to an operator
$$
\Op_h^L(b):L^2(M)\to L^2(M),
$$
see~\cite[Appendix]{meassupp}. Each of the resulting classes $S^{\comp}_{L_s,\rho}$,
$S^{\comp}_{L_u,\rho}$ satisfies standard
properties of semiclassical quantization. However, when $\rho>1/2$
and $b_1\in S^{\comp}_{L_s,\rho}$, $b_2\in S^{\comp}_{L_u,\rho}$ the product
$\Op_h^{L_s}(b_1)\Op_h^{L_u}(b_2)$ does not lie in any pseudodifferential class~-- this
observation will be important for the fractal uncertainty principle later.

As promised in the beginning of this subsection,
a version Egorov's Theorem~\eqref{e:basic-egorov}
continues to hold, see~\cite[Proposition~A.8]{meassupp}:
\begin{equation}
  \label{e:advanced-egorov}
\begin{aligned}
A_2(t)&=\Op_h^{L_s}(a_2\circ\varphi_t)+\mathcal O(h^{1-\rho-})_{L^2\to L^2},\quad 0\leq t\leq\rho\log(1/h),\\
A_2(t)&=\Op_h^{L_u}(a_2\circ\varphi_t)+\mathcal O(h^{1-\rho-})_{L^2\to L^2},\quad -\rho\log(1/h)\leq t\leq 0.\\
\end{aligned}
\end{equation}

%%%%%%%%%%%%%%%%%%%%%%%%%%%%%%%%%%%%%%%%%%%%%%%%%%%%%%%%%%%%%%%%%%%%%%%%%%%%%%%%
\subsection{Control by propagation and proof of Theorem~\ref{t:advanced-control}}
  \label{s:proof}

We now explain how to control all but a small portion of $u$
by using~\eqref{e:time-control}. Fix $0<\rho<1$ very close to 1,
to be chosen later, and choose
$$
N_1:=\lceil \rho \log(1/h)\rceil.
$$
Using the notation~\eqref{e:A-t-def}, define the operators
$$
\begin{gathered}
A_{\mathcal X}^\pm:=A_2(\mp N_1)\cdots A_2(\mp 1)A_2(0),\\ 
A_{\mathcal Y}^\pm:=\sum_{j=0}^{N_1}A_{\mathcal Y,j}^\pm,\quad
A_{\mathcal Y,j}^\pm=A_2(\mp N_1)\cdots A_2(\mp (j+1))A_1(\mp j)
(A_1+A_2)^j.
\end{gathered}
$$
Since $A_1+A_2$ commutes with $\Delta_g$, we have
$A_1(t)+A_2(t)=A_1+A_2$, therefore
$$
A_{\mathcal X}^\pm+A_{\mathcal Y}^\pm=(A_1+A_2)^{N_1+1}.
$$
Thus by~\eqref{e:part-1}
\begin{equation}
  \label{e:u-decomposed}
u=A_{\mathcal X}^\pm u+A_{\mathcal Y}^\pm u+\mathcal O(h^{1-}).
\end{equation}
Using~\eqref{e:part-1} again together with~\eqref{e:time-control}, we have
for all $j$
$$
\|A_{\mathcal Y,j}^\pm u\|\leq C\|A_1(j)u\|+\mathcal O(h^{1-})
\leq C\|\Op_h(a)u\|+\mathcal O(h^{1-}).
$$
Summing these up we see that $A_{\mathcal Y}^\pm u$ is controlled as follows:
\begin{equation}
  \label{e:Y-pm-controlled}
\|A_{\mathcal Y}^\pm u\|\leq C\log(1/h)\|\Op_h(a)u\|+\mathcal O(h^{1-}).
\end{equation}
Together with~\eqref{e:u-decomposed} this gives
$$
\|u-A_{\mathcal X}^-A_{\mathcal X}^+ u\|\leq C\log(1/h)\|\Op_h(a)u\|+\mathcal O(h^{1-}).
$$
Using the properties of $S^{\comp}_{L,\rho}$ calculus and Egorov's theorem
for long time~\eqref{e:advanced-egorov}, we get (see~\cite[Lemma~5.9]{meassupp})
$$
A_{\mathcal X}^+=\Op_h^{L_u}(a_+)+\mathcal O(h^{1-\rho-})_{L^2\to L^2},\quad
A_{\mathcal X}^-=\Op_h^{L_s}(a_-)+\mathcal O(h^{1-\rho-})_{L^2\to L^2},\quad
$$
where the symbols $a_\pm$ are given by
\begin{equation}
  \label{e:a-pm-def}
a_+=\prod_{j=0}^{N_1} (a_2\circ\varphi_{-j})\in S^{\comp}_{L_u,\rho}(T^*M\setminus 0),\quad
a_-=\prod_{j=0}^{N_1} (a_2\circ\varphi_j)\in S^{\comp}_{L_s,\rho}(T^*M\setminus 0).
\end{equation}
Thus we obtain
\begin{equation}
  \label{e:u-almost-controlled}
\|u-\Op_h^{L_s}(a_-)\Op_h^{L_u}(a_+)u\|\leq C\log(1/h)\|\Op_h(a)u\|+\mathcal O(h^{1-\rho-}).
\end{equation}
The key component of the proof is the following estimate, following from the
fractal uncertainty principle: if we take $\rho$ sufficiently close to 1 then
\begin{equation}
  \label{e:fup-used}
\|\Op_h^{L_s}(a_-)\Op_h^{L_u}(a_+)\|_{L^2\to L^2}\leq Ch^{\beta}
\end{equation}
where $\beta>0$ depends only on the set $\mathcal U$ from~\eqref{e:U-def}.%
\footnote{Note that the product $\Op_h^{L_s}(a_-)\Op_h^{L_u}(a_+)$ is not in any pseudodifferential calculus.
This makes sense in light of~\eqref{e:fup-used}: if we could write
$\Op_h^{L_s}(a_-)\Op_h^{L_u}(a_+)=\Op_h(a_-a_+)$ then the norm~\eqref{e:fup-used}
would converge to $\sup |a_-a_+|=1$ as $h\to 0$.}
Together~\eqref{e:u-almost-controlled} and~\eqref{e:fup-used} imply~\eqref{e:lazy-control},
where we take $h$ small enough to remove the $\mathcal O(h^\beta+h^{1-\rho-})$ error.

To remove the $\log(1/h)$ prefactor in~\eqref{e:lazy-control}, we have to revise the
definitions of $A_{\mathcal X}^\pm,A_{\mathcal Y}^\pm$. Roughly speaking,
we expand $A_{\mathcal X}^\pm$ to include the products
$A_{w_{N_1}}(\mp N_1)\cdots A_{w_1}(\mp 1)A_{w_0}(0)$
where $w_0,\dots,w_{N_1}\in \{1,2\}$ are such that
at most $\alpha N_1$ of the digits $w_j$ are equal to~1,
where $\alpha>0$ is a constant (the choice $\alpha=0$
would correspond to the previous definition of $A_{\mathcal X}$). If $\alpha$ is small enough
depending on $\beta$ from~\eqref{e:fup-used}, then
the norm of $A_{\mathcal X}^-A_{\mathcal X}^+$ still goes to~0
with $h$.

Correspondingly we let $A_{\mathcal Y}^\pm$ include all the products
$A_{w_{N_1}}(\mp N_1)\cdots A_{w_1}(\mp 1)A_{w_0}(0)$
where at least $\alpha N_1$ of the digits $w_j$ are equal to~1.
Using an argument similar to~\cite{Anantharaman}
(and further revising the definitions of $A_{\mathcal X}^\pm,A_{\mathcal Y}^\pm$),
we can show that for an $\alpha$-independent constant $C$ we have
the following stronger version of~\eqref{e:Y-pm-controlled}:
$$
\|A_{\mathcal Y}^\pm u\|\leq C\alpha^{-1}\|\Op_h(a)u\|+\mathcal O(h^{1/8-}),
$$
which removes the $\log(1/h)$ prefactor in~\eqref{e:lazy-control}. See~\cite[\S 4]{meassupp} for details.

%%%%%%%%%%%%%%%%%%%%%%%%%%%%%%%%%%%%%%%%%%%%%%%%%%%%%%%%%%%%%%%%%%%%%%%%%%%%%%%%
%%%%%%%%%%%%%%%%%%%%%%%%%%%%%%%%%%%%%%%%%%%%%%%%%%%%%%%%%%%%%%%%%%%%%%%%%%%%%%%%
\section{Reduction to fractal uncertainty principle}
  \label{s:fup}

In this section we explain how to prove the operator norm bound~\eqref{e:fup-used}
based on a fractal uncertainty principle, Proposition~\ref{l:fup-fourier}.

%%%%%%%%%%%%%%%%%%%%%%%%%%%%%%%%%%%%%%%%%%%%%%%%%%%%%%%%%%%%%%%%%%%%%%%%%%%%%%%%
\subsection{Porosity}

Let $a_\pm$ be the symbols defined in~\eqref{e:a-pm-def}. In this subsection
we establish \emph{porosity} of their supports in the stable/unstable directions.
We first define porosity of subsets of $\mathbb R$, which holds for
instance for fractal sets of dimension $<1$:
%%%%%%%%%%%%%%%%%%%%%%%%%%%%%%%%%%%%%%%%%%%%%%%%%%%%%%%%%%%%%%%%%%%%%%%%%%%%%%%%
\begin{defi}
  \label{d:porosity}
Let $\Omega\subset \mathbb R$ be a closed set, $0<\alpha_0\leq\alpha_1$, and $\nu\in (0,1)$.
We say that $\Omega$ is \textbf{$\nu$-porous on scales $\alpha_0$ to $\alpha_1$}
if for each interval $I$ of size $|I|\in [\alpha_0,\alpha_1]$, there exists
a subinterval $J\subset I$ such that $|J|=\nu |I|$ and $J\cap\Omega=\emptyset$.
\end{defi}
%%%%%%%%%%%%%%%%%%%%%%%%%%%%%%%%%%%%%%%%%%%%%%%%%%%%%%%%%%%%%%%%%%%%%%%%%%%%%%%%
In other words, $\Omega$ has holes at all points on all scales between $\alpha_0$ and $\alpha_1$.

The following statement shows that $\supp a_+$ is porous in the stable direction
and $\supp a_-$ is porous in the unstable direction:
%%%%%%%%%%%%%%%%%%%%%%%%%%%%%%%%%%%%%%%%%%%%%%%%%%%%%%%%%%%%%%%%%%%%%%%%%%%%%%%%
\begin{lemm}
  \label{l:porosity-1}
There exist $\nu>0$, $C_0>0$ depending only on $M,\mathcal U$ such that
for each $(x,\xi)\in T^*M\setminus 0$, the sets
$$
\Omega_\pm(x,\xi)=\{s\in\mathbb R\mid \exp(sU_\pm)(x,\xi)\in \supp a_\pm\}\subset\mathbb R
$$
are $\nu$-porous on scales $C_0h^\rho$ to $1$.
\end{lemm}
%%%%%%%%%%%%%%%%%%%%%%%%%%%%%%%%%%%%%%%%%%%%%%%%%%%%%%%%%%%%%%%%%%%%%%%%%%%%%%%%
\begin{proof}
We consider the case of $\Omega_-$. For an interval $I$ and $(x,\xi)\in T^*M\setminus 0$,
define the unstable horocyclic segment
$$
e^{IU_-}(x,\xi):=\{e^{sU_-}(x,\xi)\mid s\in I\}\subset T^*M\setminus 0.
$$
Note that the geodesic flow maps horocyclic segments to other horocyclic segments:
$$
\varphi_t(e^{IU_-}(x,\xi))=e^{(e^tI)U_-}(\varphi_t(x,\xi)).
$$
Since $\mathcal U$ is open, nonempty, and conic and the horocycle flow $\exp(sU_-)$ is uniquely ergodic
on~$S^*M$,
there exists $T=T(M,\mathcal U)\geq 1$ and $\nu>0$ such that each horocyclic segment
of length $T$ has a piece of length $10\nu T$ which lies in $\mathcal U$.
Thus for each $(x,\xi)\in S^*M$
and each interval $\widetilde I\subset\mathbb R$ with $T\leq|\widetilde I|\leq 10T$, there exists
a subinterval $\widetilde J\subset\widetilde I$ with $|\widetilde J|=\nu|\widetilde I|$
and $e^{\widetilde JU_-}(x,\xi)\subset \mathcal U$.

By~\eqref{e:a-j-def} and~\eqref{e:a-pm-def} we have for all $j=0,1,\dots,N_1$
$$
\supp a_-\cap \varphi_{-j}(\mathcal U)=\emptyset.
$$
Put $C_0:=10T$.
Assume that $(x,\xi)\in T^*M\setminus 0$ and $I\subset\mathbb R$ satisfies $C_0h^\rho\leq |I|\leq 1$.
We need to find
a subinterval $J\subset I$ with $|J|=\nu |I|$, $J\cap \Omega_-(x,\xi)=\emptyset$.

Choose $j\in \{0,1,\dots,N_1\}$ such that $T\leq e^j|I|\leq 10T$ and put
$\widetilde I:=e^j I$. Then there exists an interval $\widetilde J\subset \widetilde I$
such that $|\widetilde J|=\nu|\widetilde I|$ and
$e^{\widetilde JU_-}(\varphi_j(x,\xi))\subset \mathcal U$. Put
$J:=e^{-j}\widetilde J$, then
$J\subset I$, $|J|=\nu|I|$, and
$$
\supp a_-\cap e^{JU_-}(x,\xi)
\ \subset\ \supp a_-\cap\varphi_{-j}\big(e^{\widetilde JU_-}(\varphi_j(x,\xi))\big)
\ =\ \emptyset.
$$
Thus $J\cap\Omega_-(x,\xi)=\emptyset$ and the proof is finished.
\end{proof}
%%%%%%%%%%%%%%%%%%%%%%%%%%%%%%%%%%%%%%%%%%%%%%%%%%%%%%%%%%%%%%%%%%%%%%%%%%%%%%%%
Lemma~\ref{l:porosity-1} shows that the intersection of $\Omega_\pm$
with each stable/unstable horocycle is porous. This argument can be upgraded
to show that in fact the projection of $\Omega_\pm$ onto the stable/unstable directions
is porous, see~\cite[Lemma~5.10]{meassupp}. Splitting $a_\pm$ into finitely many pieces
via a partition of unity, we may assume that
for some $\nu>0$ depending only on $M,\mathcal U$
\begin{itemize}
\item $\supp a_\pm\subset V$ where $V\subset T^*M\setminus 0$ is a small
$h$-independent open set;
\item $\psi_\pm: V\to \mathbb R$ are smooth functions, homogeneous of order 0 and such that
$$
\ker d\psi_\pm=\Span(H_p,D,U_\mp),
$$
that is for $(x,\xi)\in V\cap S^*M$ the value $\psi_+(x,\xi)$ determines
which weak unstable leaf passes through $(x,\xi)$
and $\psi_-(x,\xi)$ determines which weak stable leaf passes through $(x,\xi)$;
\item $\supp a_\pm \subset\psi_\pm^{-1}(X_\pm)$ where $X_\pm\subset [0,1]$
are $\nu$-porous on scales $h^\rho$ to 1
(we changed the value of $\rho$ slightly to get rid of $C_0$).
\end{itemize}

%%%%%%%%%%%%%%%%%%%%%%%%%%%%%%%%%%%%%%%%%%%%%%%%%%%%%%%%%%%%%%%%%%%%%%%%%%%%%%%%
\subsection{Reduction to FUP for Fourier transform}

We now briefly explain how to reduce the norm bound~\eqref{e:fup-used} to
a one-dimensional fractal uncertainty principle for the semiclassical
Fourier transform, Proposition~\ref{l:fup-fourier}, referring to~\cite{fullgap,meassupp}
for details.

The main tool is a microlocal
normal form given by Fourier integral operators (see~\cite[\S2.2]{hgap}).
If $\varkappa:V\to T^*\mathbb R^2$ is a symplectomorphism onto its image, then
we can associate to $\varkappa$ a pair of bounded $h$-dependent operators
$$
B:L^2(M)\to L^2(\mathbb R^2),\quad
B':L^2(\mathbb R^2)\to L^2(M)
$$
such that for each $h$-independent $a\in C_0^\infty(V)$ we have
$$
\Op_h(a)=B'\Op_h(a\circ\varkappa^{-1})B+\mathcal O(h)_{L^2\to L^2}.
$$
Moreover, $BB'$ and $B'B$ are semiclassical pseudodifferential operators.
Same is true for the $S^{\comp}_{L,\rho}$ calculus, thus~\eqref{e:fup-used}
follows from the bound (suppressing the foliation $L$ in the quantization $\Op_h^L$)
\begin{equation}
  \label{e:fup-transformed}
\|\Op_h(b_-)\Op_h(b_+)\|_{L^2(\mathbb R^2)\to L^2(\mathbb R^2)}\leq Ch^{\beta},\quad
b_\pm :=a_\pm \circ\varkappa^{-1}.
\end{equation}
We would like to choose $\varkappa$ which `straightens out' the weak stable/unstable
foliations, which can be expressed in terms of the functions $\psi_\pm$.
Imagine first that we could find $\varkappa$ such that
\begin{equation}
  \label{e:this-is-unreal}
\psi_-\circ\varkappa^{-1}=x_1,\quad
\psi_+\circ\varkappa^{-1}=\xi_1.
\end{equation}
Then we get
\begin{equation}
  \label{e:b-supports}
\supp b_-\subset \{x_1\in X_-\},
\quad
\supp b_+\subset \{\xi_1\in X_+\}.
\end{equation}
We use the right quantization
$$
\Op_h(b)v(x)=(2\pi h)^{-2}\int_{\mathbb R^4} e^{{i\over h}\langle x-y,\xi\rangle}
b(y,\xi)v(y)\,dyd\xi
$$
Then the first part of~\eqref{e:b-supports} implies that for all $v\in L^2(\mathbb R^2)$
$$
\|\Op_h(b_-)v\|_{L^2}\leq C\|v\|_{L^2(X_-\times\mathbb R)}.
$$
The second one shows that $\Op_h(b_+)v$ lives at semiclassical frequencies
in $X_+\times\mathbb R$: if $\mathcal F_hf(\xi)=\widehat f(\xi/h)$ is the semiclassical
Fourier transform then
$$
\supp\mathcal F_h(\Op_h(b_+)v)\ \subset\ X_+\times\mathbb R.
$$
Therefore~\eqref{e:fup-transformed} reduces to the following estimate for all $v\in L^2(\mathbb R^2)$:
\begin{equation}
  \label{e:fup-revisited}
\supp \widehat v\subset h^{-1}X_+\times\mathbb R
\quad\Longrightarrow\quad
\|v\|_{L^2(X_-\times\mathbb R)}\leq Ch^\beta \|v\|_{L^2(\mathbb R^2)}.
\end{equation}
Since~\eqref{e:fup-revisited} does not depend on the $x_2$ variable,
we can remove it and reduce to the following one-dimensional statement
(applied with $f(x_1)=v(x_1,x_2)$ and all $x_2$, and quantifying how
close $\rho$ should be to~1):
%%%%%%%%%%%%%%%%%%%%%%%%%%%%%%%%%%%%%%%%%%%%%%%%%%%%%%%%%%%%%%%%%%%%%%%%%%%%%%%%
\begin{prop}
  \label{l:fup-fourier}
For each $\nu>0$ there exist $C,\beta>0$ such that for all $f\in L^2(\mathbb R)$,
$0<\rho\leq 1$,
and $X,Y\subset [0,1]$ which are $\nu$-porous on scales $h^\rho$ to~1,
\begin{equation}
  \label{e:fup-fourier}
\supp \widehat f\subset h^{-1}\cdot Y
\quad\Longrightarrow\quad
\|f\|_{L^2(X)}\leq Ch^{\beta-2(1-\rho)} \|f\|_{L^2(\mathbb R)}.
\end{equation}
\end{prop}
%%%%%%%%%%%%%%%%%%%%%%%%%%%%%%%%%%%%%%%%%%%%%%%%%%%%%%%%%%%%%%%%%%%%%%%%%%%%%%%%
Unfortunately there is no symplectomorphism $\varkappa$ that simultaneously
straightens out the stable and the unstable foliations, that is~\eqref{e:this-is-unreal}
can never hold. Instead one can reduce~\eqref{e:fup-used} to an estimate
of the form~\eqref{e:fup-fourier} when the Fourier transform is replaced
by a different oscillatory integral operator, see~\cite[\S5.2]{meassupp}.
The resulting statement can be reduced to the original version of~\eqref{e:fup-fourier}
(with the Fourier transform)~-- see~\cite[\S4.2]{fullgap}.

%%%%%%%%%%%%%%%%%%%%%%%%%%%%%%%%%%%%%%%%%%%%%%%%%%%%%%%%%%%%%%%%%%%%%%%%%%%%%%%%
%%%%%%%%%%%%%%%%%%%%%%%%%%%%%%%%%%%%%%%%%%%%%%%%%%%%%%%%%%%%%%%%%%%%%%%%%%%%%%%%
\section{Proof of fractal uncertainty principle}

We finally explain some ideas behind the proof of the fractal uncertainty
principle, Proposition~\ref{l:fup-fourier}. First of all we may assume
that $\rho=1$. Indeed, otherwise $X,Y$ can be written as unions
of $h^{\rho-1}$ sets which are $\nu$-porous on scales $h$ to~1
(specifically take intersections of $X,Y$ with the unions of intervals
$\bigsqcup_{j\in\mathbb Z}[h^\rho j+{h\over 2}\ell,h^\rho j+{h\over 2}(\ell+1)]$ where $\ell=0,1,\dots,\lceil 2h^{\rho-1}\rceil$),
and it remains to use the triangle inequality.

The proof of Proposition~\ref{l:fup-fourier} is given in~\cite[Theorem~4]{fullgap};
the only difference is that~\cite{fullgap} required $X,Y$ to be Ahlfors--David regular
of dimension $\delta<1$. As explained in~\cite[\S5.1]{meassupp} each $\nu$-porous
set can be embedded into a regular set of dimension $\delta=\delta(\nu)<1$.
Moreover, the proof in~\cite{fullgap} can be adapted directly to $\nu$-porous sets,
and this is the approach we take here.

%%%%%%%%%%%%%%%%%%%%%%%%%%%%%%%%%%%%%%%%%%%%%%%%%%%%%%%%%%%%%%%%%%%%%%%%%%%%%%%%
\subsection{An interation argument}

We start by exploiting the porosity of $X$ to reduce Proposition~\ref{l:fup-fourier}
to a lower bound on $f$ on a union of intervals, Proposition~\ref{l:fup-main}.
Fix large $K\in\mathbb N$ such that $2^{-K-1}\leq h\leq 2^{-K}$.
For $k\in 0,1,\dots,K$, consider the partition
$$
\mathbb R=\bigcup_{I\in\mathcal I(k)}I,\quad
\mathcal I(k):=\big\{[2^{-k}j,2^{-k}(j+1)]\colon j\in\mathbb Z\big\}.
$$
Since $X$ is $\nu$-porous on scales $h$ to~1, for each $I\in \mathcal I(k)$
there exists a subinterval
$$
I'\subset I,\quad
|I'|=\nu |I|=2^{-k}\nu,\quad
I'\cap X=\emptyset.
$$
Denote
$$
U'_k:=\bigcup_{I\in \mathcal I(k)}I'\subset\mathbb R,\quad
U'_k\cap X=\emptyset.
$$
The following is the key lower bound of $f$ on $U'_k$, proved later in this section:
%%%%%%%%%%%%%%%%%%%%%%%%%%%%%%%%%%%%%%%%%%%%%%%%%%%%%%%%%%%%%%%%%%%%%%%%%%%%%%%%
\begin{prop}
  \label{l:fup-main}
There exists $c=c(\nu)>0$ such that for all $f\in L^2(\mathbb R)$ and all $k$,
\begin{equation}
  \label{e:fup-main}
\supp \widehat f\subset h^{-1}\cdot Y+[-2^k,2^k]
\quad\Longrightarrow\quad
\|f\|_{L^2(U'_k)}\geq c\|f\|_{L^2(\mathbb R)}.
\end{equation}  
\end{prop}
%%%%%%%%%%%%%%%%%%%%%%%%%%%%%%%%%%%%%%%%%%%%%%%%%%%%%%%%%%%%%%%%%%%%%%%%%%%%%%%%
We now explain how Proposition~\ref{l:fup-main} implies Proposition~\ref{l:fup-fourier};
see~\cite[\S3.4]{fullgap} for more details.
Denote $X_k:=\mathbb R\setminus U'_k$, then $X\subset X_k$ and Proposition~\ref{l:fup-main}
implies
\begin{equation}
  \label{e:fup-main-used}
\supp \widehat f\subset h^{-1}\cdot Y+[-2^k,2^k]
\quad\Longrightarrow\quad
\|\mathbf 1_{X_k}f\|_{L^2}\leq \sqrt{1-c^2}\|f\|_{L^2}.
\end{equation}
We have
$$
\|f\|_{L^2(X)}\leq \|\mathbf 1_{X_K}\cdots\mathbf 1_{X_1}\mathbf 1_{X_0}f\|_{L^2}.
$$
We would like to use~\eqref{e:fup-main-used} to show that for all $k$,
$$
\|\mathbf 1_{X_k}\cdots\mathbf 1_{X_0}f\|_{L^2}\leq \sqrt{1-c^2}\|\mathbf 1_{X_{k-1}}\cdots\mathbf 1_{X_0}f\|_{L^2}.
$$
This works well for $k=0$, since $\supp \widehat f\subset h^{-1}\cdot Y$. However,
the next values of $k$ do not work since we have no
restriction on the Fourier support of $\mathbf 1_{X_{k-1}}\cdots\mathbf 1_{X_0}f$.
To fix this we choose large $k_0$ depending on $c,\nu$ and
replace $\mathbf 1_{X_k}$ by its mollification
$$
\chi_k:=\mathbf 1_{X_k}*\psi_k,\quad
\psi_k(x)=2^{k+k_0}\psi(2^{k+k_0}x)
$$
where $\psi$ is a nonnegatives Schwartz function with
$\supp\widehat\psi\subset [-{1\over 2},{1\over 2}]$,
and $\int\psi=1$.
Note that $0\leq\chi_k\leq 1$ and $\supp\widehat\chi_k\subset\supp\widehat\psi_k\subset [-2^{k+k_0-1},2^{k+k_0-1}]$.

For $k$ divisible by $k_0$, put $f_k:=\chi_k\chi_{k-k_0}\cdots\chi_{k_0}\chi_0 f$. Then
$$
\supp \widehat f_{k-k_0}\ \subset\ \supp\widehat\chi_{k-k_0}+\cdots+\supp\widehat\chi_0+\supp\widehat f
\ \subset\ h^{-1}\cdot Y+[-2^k,2^k].
$$
We can arrange to have $\chi_k\leq {1\over 2}$ on $U'_k$.
Then \eqref{e:fup-main-used} applies to~$f_{k-k_0}$, implying
\begin{equation}
  \label{e:fup-iterating}
\|f_k\|_{L^2}= \|\chi_k f_{k-k_0}\|_{L^2}\leq \sqrt{1-c^2/10}\,\|f_{k-k_0}\|_{L^2}.
\end{equation}
Shrinking the intervals $I'$ by a factor of 2, we can ensure
that $X_k$ contains the $2^{-k-2}\nu$-neighborhood of $X$.
Then for large $k_0$, we have
$\chi_k\geq 1-2^{-k_0}$ on $X$. Iterating~\eqref{e:fup-iterating} we get
(assuming for simplicity that $K$ is divisible by $k_0$)
$$
\|f\|_{L^2(X)}\leq (1-2^{-k_0})^{-K/k_0-1}\|f_K\|_{L^2}
\leq 2\Big({\sqrt{1-c^2/10}\over 1-2^{-k_0}}\Big)^{K/k_0}\|f\|_{L^2}.
$$
Taking $k_0$ large enough, we can make this bounded
by $(1-c^2/100)^{K/k_0}\|f\|_{L^2}$. Since $c,k_0$ are independent
of $h$ and $K\sim\log(1/h)$ this
gives~\eqref{e:fup-fourier}
and finishes the proof of Proposition~\ref{l:fup-fourier}.

%%%%%%%%%%%%%%%%%%%%%%%%%%%%%%%%%%%%%%%%%%%%%%%%%%%%%%%%%%%%%%%%%%%%%%%%%%%%%%%%
\subsection{Unique continuation estimates}

We now discuss how to prove Proposition~\ref{l:fup-main}, which is a \emph{unique continuation estimate}
for functions of restricted Fourier support.

It suffices to consider the case $k=0$. Indeed, otherwise we define
the rescaled function $\widetilde f(x):=2^{-k/2}f(2^{-k}x)$.
The Fourier transform of $\widetilde f$ is supported on
$2^{-k}\supp\widehat f\subset \widetilde h^{-1}\cdot Y+[-1,1]$
where $\widetilde h:=2^kh$. Moreover
$\|f\|_{L^2(U'_k)}=\|\widetilde f\|_{L^2(2^kU'_k)}$
and $2^kU'_k$ satisfies same assumptions as $U'_0$.
It remains to apply~\eqref{e:fup-main} with $k:=0$, $h:=\tilde h$,
and $f:=\widetilde f$. We henceforth assume that $k=0$ and
write $\mathcal I:=\mathcal I(0)$ and $U':=U'_0$.

In this section we show a bound of the type~\eqref{e:fup-main} which replaces
the requirement on $\supp\widehat f$ with a Fourier decay condition.
See~\cite[\S\S3.2,3.3]{fullgap} for details.
%%%%%%%%%%%%%%%%%%%%%%%%%%%%%%%%%%%%%%%%%%%%%%%%%%%%%%%%%%%%%%%%%%%%%%%%%%%%%%%%
\begin{lemm}
  \label{l:harmonic}
Assume that $\theta$ is a positive even function on $\mathbb R$,
$\theta$ is decreasing on $[0,\infty)$, and $\xi\theta(\xi)$ is increasing
on $[0,\infty)$.
Then there exists a constant $C$ depending only
on~$\nu,\theta$ such that for all $K\geq 10$
and all $g\in L^2(\mathbb R)$ we have
\begin{equation}
  \label{e:harmonic}
\|g\|_{L^2(\mathbb R)}\leq {C\over \theta(K)}\|g\|_{L^2(U')}^{\kappa}\cdot
\big\|e^{|\xi|\theta(\xi)}\widehat g(\xi)\big\|_{L^2(\mathbb R)}^{1-\kappa}
+{Ce^{-\kappa K\theta(K)}\over\theta(K)}\big\|e^{|\xi|\theta(\xi)}\widehat g(\xi)\big\|_{L^2(\mathbb R)}
\end{equation}
where $\kappa=e^{-C/\theta(K)}$.
\end{lemm}
%%%%%%%%%%%%%%%%%%%%%%%%%%%%%%%%%%%%%%%%%%%%%%%%%%%%%%%%%%%%%%%%%%%%%%%%%%%%%%%%
\Remark Letting $K\to\infty$ we see that if
\begin{equation}
  \label{e:condition-theta}
e^{-C/\theta(K)}K\theta(K)-\log\theta(K)\to\infty\quad\text{as }K\to\infty
\end{equation}
then any $g\not\equiv 0$ with $e^{|\xi|\theta(\xi)}\widehat g(\xi)\in L^2$
cannot identically vanish on $U'$; this can be used to show that such~$g$ cannot be compactly supported.
Note that~\eqref{e:condition-theta} holds when $\theta(K)=1$,
this is reasonable since functions with exponentially decaying Fourier transform are real analytic.
On the other hand~\eqref{e:condition-theta} does not hold when $\theta(K)=K^{-\varepsilon}$
and $\varepsilon>0$; this corresponds to existence of compactly supported functions in Gevrey classes.
We will use later that~\eqref{e:condition-theta} holds
with $\theta(K)=(\log K)^{-\delta}$ when $\delta<1$.
%%%%%%%%%%%%%%%%%%%%%%%%%%%%%%%%%%%%%%%%%%%%%%%%%%%%%%%%%%%%%%%%%%%%%%%%%%%%%%%%
\begin{proof}[Sketch of the proof]
\newcommand{\high}{\mathrm{high}}
\newcommand{\low}{\mathrm{low}}
Put $R:=\|e^{|\xi|\theta(\xi)}\widehat g(\xi)\|_{L^2}$.
We split $g$ into low and high frequencies:
$$
g=g_{\low}+g_{\high},\quad
\supp \widehat g_{\low}\subset [-K,K],\quad
\supp \widehat g_{\high}\subset \mathbb R\setminus [-K,K].
$$
Then $g_{\high}$ is estimated in terms of $R$:
\begin{equation}
  \label{e:f-high}
\|g_{\high}\|_{L^2}\leq e^{-K\theta(K)}R.
\end{equation}
As for $g_{\low}$, we claim that it satisfies the bound
\begin{equation}
  \label{e:harmonic-1}
\|g_{\low}\|_{L^2(\mathbb R)}\leq {C\over \theta(K)}\|g_{\low}\|_{L^2(U')}^\kappa\cdot R^{1-\kappa}.
\end{equation}
We only show the following local version: for all $I\in\mathcal I$,
\begin{equation}
  \label{e:harmonic-2}
\|g_{\low}\|_{L^2(I)}\leq {C\over \theta(K)}\|g_{\low}\|_{L^2(I')}^\kappa\cdot R^{1-\kappa}.
\end{equation}
The bound~\eqref{e:harmonic-1} is obtained similarly, summing
over all the intervals $I$ in the end;
see the proof of~\cite[Lemma~3.2]{fullgap}.

To prove~\eqref{e:harmonic-2}, we extend $g_{\low}$ holomorphically to $\mathbb C$;
this is possible since $\widehat g_{\low}$ is compactly supported.
Consider the following domain in the complex plane:
$$
\Sigma:=\{z\in \mathbb C\colon |\Im z|<r\}\setminus I',\quad
r:=\theta(K).
$$
For each $t\in I\setminus I'$, let $\mu_t$ be the harmonic measure of $\Sigma$
centered at $t$; it is a probability measure on the boundary
$$
\partial\Sigma=\Sigma_+\sqcup\Sigma_-\sqcup I',\quad
\Sigma_\pm:=\{\Im z=\pm r\}.
$$
In other words, for each harmonic function $u$ on $\Sigma$ its value $u(t)$ is equal to the integral of
the boundary values of $u$ over $\mu_t$. Since $\log|g_{\low}|$ is subharmonic, we have
\begin{equation}
  \label{e:bunny-1}
\log |g_\low(t)|\leq \int_{\partial\Sigma}\log|g_{\low}(z)|\,d\mu_t(z).
\end{equation}
We have $\mu_t(I')\geq \kappa$ for all $t\in I\setminus I'$. This can be explained for instance
using the stochastic interpretation of harmonic measure: $\mu_t$ is the probability distribution
of the point through which the Brownian motion starting at $t$ will exit the domain $\Sigma$.
Since $t$ is distance 1 away from $I'$ and $\Sigma$ has height $2r$, the
probability that the Brownian motion starting at $t$ hits $I'$ before $\Sigma_\pm$
is bounded below by $e^{-C/r}$.

Now~\eqref{e:bunny-1} implies (using that $\sup_{I'} |g_{\low}|\leq \sup_{\Sigma_+\sqcup\Sigma_-}|g_{\low}|$)
\begin{equation}
  \label{e:bunny-2}
\sup_I|g_{\low}|\leq \big(\sup_{I'}|g_{\low}|\big)^\kappa\cdot \big(\sup_{\Sigma_+\sqcup\Sigma_-}|g_{\low}|\big)^{1-\kappa}.
\end{equation}
Using estimates on the density of $\mu_t$ one can replace the sup-norms in~\eqref{e:bunny-2}
by $L^2$-norms, with small changes~-- see the proof of~\cite[Lemma~3.2]{hgap}.
This yields~\eqref{e:harmonic-2}, where we use that
$$
\|g_{\low}\|_{L^2(\Sigma_\pm)}\leq \|e^{r|\xi|}\widehat g_{\low}(\xi)\|_{L^2(\mathbb R)}\leq R.
$$

Armed with~\eqref{e:f-high} and~\eqref{e:harmonic-1} we write
$$
\begin{aligned}
\|g\|_{L^2}&\leq \|g_{\low}\|_{L^2}+\|g_{\high}\|_{L^2}
\leq {C\over \theta(K)}\|g_{\low}\|_{L^2(U')}^\kappa\cdot R^{1-\kappa}
+\|g_{\high}\|_{L^2}\\
&\leq
{C\over \theta(K)}\big(\|g\|_{L^2(U')}^\kappa+\|g_{\high}\|_{L^2}^\kappa\big)\cdot R^{1-\kappa}
+\|g_{\high}\|_{L^2}\\
&\leq
{C\over\theta(K)}\|g\|_{L^2(U')}^\kappa\cdot R^{1-\kappa}
+{Ce^{-\kappa K\theta(K)}\over\theta(K)}R
\end{aligned}
$$
which gives~\eqref{e:harmonic}.
\end{proof}
%%%%%%%%%%%%%%%%%%%%%%%%%%%%%%%%%%%%%%%%%%%%%%%%%%%%%%%%%%%%%%%%%%%%%%%%%%%%%%%%

%%%%%%%%%%%%%%%%%%%%%%%%%%%%%%%%%%%%%%%%%%%%%%%%%%%%%%%%%%%%%%%%%%%%%%%%%%%%%%%%
\subsection{An adapted multiplier and end of the proof}

We finally explain how to use the Fourier support condition $\supp\widehat f\subset h^{-1}\cdot Y+[-1,1]$.
To do that we construct a compactly supported multiplier adapted to $Y$:
%%%%%%%%%%%%%%%%%%%%%%%%%%%%%%%%%%%%%%%%%%%%%%%%%%%%%%%%%%%%%%%%%%%%%%%%%%%%%%%%
\begin{lemm}
  \label{l:multiplier}
There exists constants%
\footnote{In the notation of~\cite{fullgap}, one should replace $\delta$
by $1+\delta\over 2$.}
$c_1>0$, $\delta<1$ depending only on $\nu$ and a function
$\psi\in L^2(\mathbb R)$ such that, defining
$$
\theta(\xi):=\big(\log(10+|\xi|)\big)^{-\delta},
$$
we have
\begin{enumerate}
\item $\supp\psi\subset [-\nu/10,\nu/10]$;
\item $\|\widehat\psi\|_{L^2(-1/2,1/2)}\geq c_1$;
\item $|\widehat\psi(\xi)|\leq (1+|\xi|)^{-10}\exp(-c_1|\xi|\theta(\xi))$ for all $\xi\in h^{-1}\cdot Y+[-2,2]$;
\item $|\widehat\psi(\xi)|\leq (1+|\xi|)^{-10}$ for all $\xi$.
\end{enumerate}
\end{lemm}
%%%%%%%%%%%%%%%%%%%%%%%%%%%%%%%%%%%%%%%%%%%%%%%%%%%%%%%%%%%%%%%%%%%%%%%%%%%%%%%%
Before explaining the proof of Lemma~\ref{l:multiplier} we use it to give
%%%%%%%%%%%%%%%%%%%%%%%%%%%%%%%%%%%%%%%%%%%%%%%%%%%%%%%%%%%%%%%%%%%%%%%%%%%%%%%%
\begin{proof}[Sketch of the proof of Proposition~\ref{l:fup-main}]
For each $I\in\mathcal I$, let $I''\subset I'$ be the interval with the same
center and $|I''|={1\over 2}|I'|=\nu/2$. Let $U''$ be the union of all $I''$.
Take the function $\psi$ from Lemma~\ref{l:multiplier} and define
$$
g:=f * \psi,\quad
\widehat g=\widehat f\cdot\widehat \psi.
$$
Since $\supp \widehat f\subset h^{-1}\cdot Y+[-1,1]$ we have (using the Sobolev space $H^{-10}(\mathbb R)$)
$$
\|e^{c_1|\xi|\theta(\xi)}\widehat g(\xi)\|_{L^2}\leq \|f\|_{H^{-10}}.
$$
Applying Lemma~\ref{l:harmonic} with $U''$ in place of $U'$, we get for all $K$
\begin{equation}
  \label{e:eagle-1}
\|g\|_{L^2(\mathbb R)}\leq {C\over \theta(K)}\|g\|_{L^2(U'')}^\kappa\cdot \|f\|_{H^{-10}}^{1-\kappa}
+{Ce^{-c_1\kappa K\theta(K)}\over\theta(K)}\|f\|_{H^{-10}}.
\end{equation}
Now, the support condition on $\psi$ implies that
$g=(\mathbf 1_{U'}f)*\psi$ on $U''$, therefore
$\|g\|_{L^2(U'')}\leq \|\mathbf 1_{U'}f\|_{H^{-10}}$. We can apply the same argument
to $f_\eta(x)=e^{ix\eta}f(x)$, $|\eta|\leq 1$, since
$\supp \widehat f_\eta\subset h^{-1}\cdot Y+[-2,2]$.
Then~\eqref{e:eagle-1} bounds
$\|f_\eta*\psi\|_{L^2}=\|\widehat f(\xi-\eta)\widehat\psi(\xi)\|_{L^2}$.
Integrating the square of this over $\eta$ and using
property~(2) of $\psi$ we bound
$\|\widehat f\|_{L^2(-1/2,1/2)}$:
\begin{equation}
  \label{e:eagle-2}
\|\widehat f\|_{L^2(-1/2,1/2)}\leq {C\over \theta(K)}\|\mathbf 1_{U'}f\|_{H^{-10}}^\kappa\cdot \|f\|_{H^{-10}}^{1-\kappa}
+{Ce^{-c_1\kappa K\theta(K)}\over\theta(K)}\|f\|_{H^{-10}}.
\end{equation}
Finally, we see that~\eqref{e:eagle-2} holds also when
$f$ is replaced by $f_\ell(x)=e^{i\ell x}f(x)$ and $\ell\in\mathbb Z$, $|\ell|\leq h^{-1}$.
Indeed, $\supp\widehat f_\ell$ lies in $h^{-1}(Y+\ell h)+[-1,1]$ and
$Y+\ell h$ is still $\nu$-porous.%
\footnote{This means that the argument actually uses many multipliers $\psi$,
one for each value of $\ell$.}
Adding the resulting inequalities (see the proof of~\cite[Proposition~3.3]{fullgap}) we get
\begin{equation}
  \label{e:eagle-3}
\|f\|_{L^2}\leq {C\over \theta(K)}\|f\|_{L^2(U')}^\kappa\cdot \|f\|_{L^2}^{1-\kappa}
+{Ce^{-c_1\kappa K\theta(K)}\over\theta(K)}\|f\|_{L^2}.
\end{equation}
As remarked after Lemma~\ref{l:harmonic}, the function $\theta(K)=c_1(\log (10+K))^{-\delta}$
satisfies~\eqref{e:condition-theta}. Therefore if $K$ is large (but independent of $h$)
the last term on the right-hand side of~\eqref{e:eagle-3} can be removed,
yielding~\eqref{e:fup-main}.
\end{proof}
%%%%%%%%%%%%%%%%%%%%%%%%%%%%%%%%%%%%%%%%%%%%%%%%%%%%%%%%%%%%%%%%%%%%%%%%%%%%%%%%
We finally explain how to prove Lemma~\ref{l:multiplier}.
We ignore the $|\xi|^{-10}$ prefactor in property~(3)
(one can always convolve $\psi$ with a function in $C_0^\infty(\mathbb R)$ to regain it).
We also replace $h^{-1}\cdot Y+[-2,2]$ with just $h^{-1}\cdot Y$ for simplicity, but the argument below applies to $h^{-1}\cdot Y+[-2,2]$ as well.

Denote $\widetilde Y:=h^{-1}\cdot Y$.
We first construct a `reasonable' weight function
\begin{equation}
  \label{e:weight-as}
w:\mathbb R\to [0,\infty),\quad
w(\xi)\geq |\xi|\theta(\xi)\quad\text{for }\xi\in \widetilde Y.
\end{equation}
Assume for simplicity that $h=2^{-K}$ and consider a dyadic partition of $\widetilde Y$:
$$
\widetilde Y=\bigcup_{k=0}^K (\widetilde Y\cap J_k),\quad
J_0=[-1,1],\quad
J_k=[-2^k,-2^{k-1}]\cup[2^{k-1},2^k]\text{ for }k\geq 1.
$$
We next take a minimal covering of each $\widetilde Y\cap J_k$ by intervals of size $2^k\theta(2^k)$:
$$
\widetilde Y\cap J_k\subset \bigcup_{\ell=1}^{N_k} J_{k\ell},\quad
|J_{k\ell}|=2^k\theta(2^k).
$$
It is finally time to use the fact that $Y$ is $\nu$-porous:
we claim that there exists $\varepsilon=\varepsilon(\nu)>0$
such that
\begin{equation}
  \label{e:N-bound}
N_k\leq C \big(\theta(2^k)\big)^{-(1-\varepsilon)}.
\end{equation}
Indeed, put $m:=\lceil 2/\nu\rceil$. Divide $J:=[2^{k-1},2^k]$ into
$m$ intervals of size $m^{-1}|J|$ each. It follows from
$\nu$-porosity of $Y$ that at least one of these intervals does not intersect
$\widetilde Y$. Divide each of the remaining intervals into $m$ intervals of size
$m^{-2}|J|$ each and repeat the process, at $j$-th step
keeping only $(m-1)^j$ intervals of size $m^{-j}|J|$ each. We stop once $m^{-j}\sim\theta(2^k)$,
then the number of resulting intervals is bounded by~\eqref{e:N-bound}
where $\varepsilon=1-{\log(m-1)\over \log m}$.

Now, define nonnegative weights $w_{k\ell}\in C_0^\infty(\mathbb R)$ such that $\sup|w'_{k\ell}|\leq 100$,
$w_{k\ell}$ is supported on the $2^k\theta(2^k)$-neighborhood of $J_{k\ell}$,
and $w_{k\ell}=2^k\theta(2^k)$ on $J_{k\ell}$. Put
$$
w:=10\sum_{k=0}^K \sum_{\ell=1}^{N_k} w_{k\ell},
$$
then $w$ satisfies~\eqref{e:weight-as} for large enough $|\xi|$ and one can change
it for $|\xi|\leq C$ to make sure~\eqref{e:weight-as} holds everywhere.  

Given~\eqref{e:weight-as}, it remains to construct $\psi$ which
satisfies properties~(1) and~(2) in Lemma~\ref{l:multiplier}
and $|\widehat\psi|\leq \exp(-c_1w)$. One way is to use the Beurling--Malliavin Theorem~\cite{Beurling-Malliavin},
see~\cite[\S3.1]{fullgap}. This theorem asserts that there exists
a function
$$
\psi\in L^2(\mathbb R),\quad
\psi\not\equiv 0,\quad
|\psi|\leq \exp(-w)
$$
if the derivative $w'$ is bounded and the Poisson integral of $w$ converges:
\begin{equation}
  \label{e:w-integral}
\int_{\mathbb R}{w(\xi)\over 1+\xi^2}\,d\xi\leq \infty.
\end{equation}
Our weight $w$ has derivative bounded by an $h$-independent constant.
Same is true of its Poisson integral: using~\eqref{e:N-bound} we estimate
$$
\int_{\mathbb R}{w(\xi)\over 1+\xi^2}\,d\xi\leq C\sum_{k=0}^K N_k \big(\theta(2^k)\big)^2
\leq C\sum_{k=0}^K \big(\theta(2^k)\big)^{1+\varepsilon}
\leq C\sum_{k=0}^K k^{-\delta(1+\varepsilon)}
$$
which is bounded if we choose $\delta<1$ such that $\delta(1+\varepsilon)>1$.
A compactness argument using the Beurling--Malliavin theorem
(see~\cite[Lemma~2.11]{fullgap}) now gives Lemma~\ref{l:multiplier}.

The Beurling--Malliavin theorem is one of the deepest statements in harmonic analysis.
It is thus worth noting that we do not need the full strength of this theorem here.
This is due to the fact that the Hilbert transform of our weight $w$ has derivative
bounded by some $h$-independent constant, and there is a direct construction of the
function $\psi$ for such weights. See~\cite{JinZhang} for details.

%%%%%%%%%%%%%%%%%%%%%%%%%%%%%%%%%%%%%%%%%%%%%%%%%%%%%%%%%%%%%%%%%%%%%%%%%%%%%%%%
%%%%%%%%%%%%%%%%%%%%%%%%%%%%%%%%%%%%%%%%%%%%%%%%%%%%%%%%%%%%%%%%%%%%%%%%%%%%%%%%
\medskip\noindent\textbf{Acknowledgements.}
This research was conducted during the period the author served as
a Clay Research Fellow.

%%%%%%%%%%%%%%%%%%%%%%%%%%%%%%%%%%%%%%%%%%%%%%%%%%%%%%%%%%%%%%%%%%%%%%%%%%%%%%%%

\end{document}